\documentclass[11pt]{amsart}

\usepackage[utf8]{inputenc}
\usepackage[T1]{fontenc}
\usepackage[english]{babel}
\usepackage{graphicx}
\usepackage{amsmath}
\usepackage{amssymb}
\usepackage{amsthm}
\usepackage{fullpage}
\usepackage{mathrsfs}
\usepackage{mathtools}
\usepackage{tikz-cd}
\usepackage[hidelinks]{hyperref}
\usepackage{amsaddr}
\usepackage{cleveref}
\usepackage{mathabx}

\hypersetup{allcolors=black}

\newtheorem{theorem}{Theorem}
\numberwithin{theorem}{section}
\newtheorem{lemma}[theorem]{Lemma}
\newtheorem{proposition}[theorem]{Proposition}
\newtheorem{corollary}[theorem]{Corollary}

\newtheorem{definition}[theorem]{Definition}
\newtheorem{remark}{Remark}
\numberwithin{remark}{section}

\numberwithin{example}{section}
\newtheorem*{theoremA*}{Theorem A}
\newtheorem*{theoremB*}{Theorem B}
\newtheorem*{question*}{Question}

\newcommand\C{\mathbb C}
\newcommand\N{\mathbb N}

\newcommand\Bb{\mathcal B}

\newcommand\Ff{\mathcal F}

\newcommand\Jj{\mathcal J}

\newcommand\Mm{\mathcal M}

\newcommand\Pp{\mathcal P}
\newcommand\Qq{\mathcal Q}

\newcommand\id{\text{id}}
\newcommand\Supp{\text{Supp}}
\newcommand\qi{\simeq_q}
\DeclareMathOperator\Adj{Adj}
\DeclareMathOperator\Aut{Aut}
\DeclareMathOperator\co{co}
\DeclareMathOperator\Diag{Diag}
\DeclareMathOperator\Fix{Fix}
\DeclareMathOperator\Jor{Jor}
\DeclareMathOperator\Qu{Qu}
\DeclareMathOperator\Stab{Stab}
\DeclarePairedDelimiter{\abs}{\lvert}{\rvert}
\makeatletter
\let\oldabs\abs
\def\abs{\@ifstar{\oldabs}{\oldabs*}}
\title{Noncommutative properties of 0-hyperbolic graphs}
\author{Amaury Freslon$ ^{1} $, Paul Meunier$ ^{2} $, Pegah Pournajafi$ ^{3} $}
	
\thanks{\noindent $ ^{1} $Amaury Freslon, Université Paris-Saclay, CNRS, Laboratoire de Mathématiques d’Orsay, 91405 Orsay, France \\ 
\indent $ ^{2}$Paul Meunier, KU Leuven, Department of Mathematics, Celestijnenlaan 200B box 2400, BE-3001 Leuven, research supported by the grant 11PAL24N funded by the Research Foundation Flanders (FWO) \\ 
\indent $ ^{3} $Pegah Pournajafi, Chaire Combinatoire, Collège de France, Université PSL, 75005, Paris, France}

\begin{document}
	
\vspace*{-.55cm}
\maketitle

\begin{abstract}
	We study several noncommutative properties of 0-hyperbolic graphs. 
	In particular, we prove that 0-hyperbolicity is preserved under quantum isomorphism. 
	We also compute the quantum automorphism groups of 0-hyperbolic graphs and characterise the ones with quantum symmetry.
\end{abstract}

\section{Introduction}\label{sec:intro}

The study of noncommutative properties of graphs has become a topic of great interest at the intersection of graph theory and quantum group theory. 
It all starts with a simple question: what happens when one studies the representation theory of a graph in a noncommutative algebra? 
This apparently modest artefact surprisingly unravels formerly hidden deeper properties of graphs, ones which vanish when represented in commutative algebras. 
These noncommutative properties, far from being disorganised, follow intricate and subtle strucutres. 
For instance, while the classical symmetries of a graph form a group, their noncommutative symmetries form a quantum group, the natural noncommutative dual of groups. 
Moreover, the notion of isomorphism of graphs evolves into the quantum isomorphism of graphs. 
These notions are neither simple -- the quantum isomorphism of graphs becomes an undecidable property~\cite{Mancinska2019QiNotI}, and encodes the representation theory of finitely presented groups~\cite{Slofstra2020} -- nor arbitrary -- the same notion of quantum isomorphism still preserves many core properties of graphs, such as their spectrum or their block structures~\cite{freslon2025blockstructures}.

The key point of the present paper is to consider the fact that quantum isomorphism preserves in particular many metric properties of graphs, in the line of the work initally started by Fulton~\cite{FultonThesis}. 
It is now well-known that connectedness, and recently 2-connectedness~\cite{freslon2025blockstructures} are preserved, as well as many metric invariants, such as the centre, the radius, or the diameter of a graph. 

One can then wonder what happens with the hyperbolicity of a graph. 
Here, the hyperbolicity we consider is the one obtained from the Gromov product, as defined in~\cite{Gromov1987hyperbolicgroups}. 

We obtain the following result.

\begin{theoremA*}\label{thm:introA}
	Let $G$ be a 0-hyperbolic graph, and let $H$ be quantum isomorphic to $G$. 
	Then $H$ is isomorphic to $G$, in particular, it is also 0-hyperbolic.
\end{theoremA*}

This suggests that it is possible to fully understand the noncommuative properties of 0-hyperbolic graphs from their classical properties, in the sense of~\cite{Meunier2023}. 
This is the main result of the current paper. 
The family of 0-hyperbolic graphs has a purely structural characterisation: they are exactly the graphs whose blocks are complete graphs. 
That is why they are also called block graphs, and that is the name we will primarily use in this article.

\begin{theoremB*}\label{thm:introB}
	The family of block graphs satisfies the following properties:
	\begin{itemize}
		\item Any graph quantum isomorphic to a block graph is isomorphic it;
		\item A block graph has quantum symmetry if and only if it has two nontrivial automorphisms with disjoint support;
		\item The quantum automorphism groups of block graphs form the class $\Jor^+(\C)$, that is, the Jordan closure of the trivial compact quantum group.
	\end{itemize}
\end{theoremB*}

This generalises similar results from~\cite{Meunier2023} about trees (for the last item, see also~\cite{vanDobbenetal2023}).

We conclude this introduction by mentioning a natural next step. 
Let $\delta(G)$ denote the hyperbolicity of $G$.

\begin{question*}\label{qu:hyperbolicity}
	Are there two quantum isomorphic graphs $G$ and $H$ such that $\delta(G)\neq \delta(H)$, and if so, what are the minimum hyperbolicities of such graphs?
\end{question*}

Let us briefly outline the organisation of the paper. 
We start with preliminary definitions and results in~\Cref{sec:preliminaries}. 
After this, we introduce general theoretical tools such as the quantum automorphism group of partitioned graphs in~\Cref{sec:partitioned_graphs}, obtain a decomposition of the quantum group of a graph in~\Cref{sec:psi}, and prove some useful graph-theoretical results in~\Cref{sec:complete_blocks}. 
This allows us to study block graphs in depth in~\Cref{sec:block_graphs}.

\section{Preliminaries}\label{sec:preliminaries}

This section contains reminders of known notation and results, as well as some new results that will be used throughout the paper.
We mostly use the same setting and notation as~\cite{freslon2025blockstructures}, which we now briefly recall.

\subsection{Graph theory} \label{sec:pre-graphtheory}

For any graph theoretical notion not defined here, we refer to~\cite{BondyMurty}. 
All graphs in this paper are finite and simple -- that is, loopless and without multiple edges. 
We denote the vertex set and the edge set of a graph $ G $ respectively by $ V(G) $ and $ E(G) $, and its adjacency matrix by $\Adj(G)$. 
We write $ xy $ for an element $ \{x,y\} \in E(G) $. 
A \emph{graph morphism} (or simply a \emph{morphism}) from a graph $ G $ to a graph $ H $ is a function $ \phi\colon V(G) \to V(H) $ such that for all $ \{x,y\} \in E(G) $ we have $ \phi(x)\phi(y) \in E(H) $. 
A morphism $ \phi $ is an isomorphism if and only if it is a bijection and for every $ x, y \in V(G) $ such that $ xy\notin E(G) $ we have $ \phi(x)\phi(y) \notin E(H) $. 
We denote by $G[S]$ the subgraph of $G$ induced by a set of vertices $S\subseteq V(G)$, that is, $V(G[S]) = S$ and $ E(H) = \{ \{x,y\} \in E(G) \mid x,y \in S \} $. 
A \emph{hereditary class of graphs} (or simply a \emph{class} of graphs) is a family of graphs that is closed under taking induced subgraphs.  We denote the disjoint union of two graphs $ G $ and $ H $ by $ G \oplus H $. We denote $G \oplus G \oplus \dots \oplus G$ by $k.G$ where $k$ is the number of factors in the sum.

A graph $ G $ is said to be connected if for every $ x,y \in V(G) $, there exists a walk from $ x $ to $ y $ in $ G $, that is to say a sequence $v_1=x,\ldots,v_n=y$ of vertices such that $v_iv_{i+1}\in E(G)$ for all $1\leq i\leq n-1$. 
The \emph{connected components} of $ G $ are the maximal connected subgraphs of $ G $. 
If $ G $ has at least two vertices, a \emph{cut vertex} of $ G $ is a vertex $ v $ such that $ G \setminus \{v\} $ has strictly more connected components than $ G $. 
For a graph on one vertex, in this paper, we define its only vertex to be a cut vertex. 

We recall that every graph in this paper is finite, loopless, and without multiple edges. 
In this setting, a graph $ G $ is \emph{2-connected} if for every $ v \in V(G) $, the graph $ G \setminus \{v\} $ is connected. 
A maximal 2-connected subgraph of $ G $ is called a \emph{block} of $ G $. 
We call a vertex $ v $ in a block $ B $ of $ G $ an \emph{internal vertex} if it is not contained in any block of $ G $ but $ B $. 
This paper is concerned with a class of graphs for which blocks have a specific form.

\begin{definition}\label{def:block_graph}
	A \emph{block graph} is a graph all of whose blocks are complete graphs.
\end{definition}

\begin{remark}\label{rem:bg_of_a_graph}
	To every graph $G$, one can associate the intersection graph $\Bb(G)$ of its blocks, which is called the \emph{block graph of $G$}. 
	It turns out (see~\cite[Theorem A]{Harary1963blockgraphs}) that a graph is of the form $\Bb(G)$ if and only if it is a block graph in the sense of the definition above, hence the terminology.
\end{remark}

Every graph $ G $ is equipped with a distance
\[ d =d_G\colon V(G) \times V(G) \to [0, +\infty]\]
called the \emph{graph distance} where for $ x, y \in V(G) $, the distance $ d(v,u) $ is the length of a shortest path from $v$ to $u$ in $G$. 

Let $w,x,y,z \in V(G)$ and set $d_1=d(w,x) + d(y,z)$, $d_2=d(w,y)+d(x,z)$, and $d_3=d(w,z)+d(x,y)$. Let $\pi \in S_3$ be a permutation such that $d_{\pi(1)} \geq d_{\pi(2)} \geq d_{\pi(3)}$. Define $h(w,x,y,z)=d_{\pi(1)}-d_{\pi(2)}$. 

The \emph{hyperbolicity} of $ G $, denoted by $ \delta(G) $, is defined as follows:
$$\delta(G)=\frac{1}{2} \max_{w,x,y,z\in V(G)} h(w,x,y,z).$$

It is worth noting that there are multiple notions of hyperbolicity. This one is the definition of hyperbolicity with the Gromov product in~\cite{Gromov1987hyperbolicgroups}, applied to the graph metric. 

The following folklore theorem characterises finite 0-hyperbolic graphs. 

\begin{theorem}
	A connected graph is 0-hyperbolic if and only if it is a block graph.
\end{theorem}

From now on, we use the term ``block graphs'' to refer to this class. 

\medskip

For every vertex $ v \in V(G) $, the \emph{eccentricity} of $ v $, denoted by $ e(v) $, is defined to be the its distance to the vertex furthest from it, that is $e(v) = \max_{u \in V(G)} d(v,u) $. 
We denote by $ E(v) $ the set $  \{u \in V(G) \mid  d(v,u) = e(v) \} $ and call each element of $E(v) $ an \emph{eccentric vertex} for $ v $.

The \emph{centre} of $ G $, denoted by $Z(G)$, is the set of vertices with minimum eccentricity, that is
\[ Z(G) = \{v \in V(G) \mid \forall u \in V(G) \ e(u) \geq e(v) \}.\]

The centre is related to the block structure thanks to the following (see~\cite[Theorem 2.2]{BuckleyHarary1990Distancesingraphs}).

\begin{theorem}\label{thm:centre_block}
	Let $G$ be a connected graph. 
	Then, either $Z(G)$ is a cut vertex or it is contained in a unique block $B_{Z(G)}$. 
\end{theorem}

\subsection{Quantum permutation groups} \label{subsec:QPG}

One of our main goals is to describe the quantum symetries of block graphs, which are encoded in a structure called its \emph{quantum automorphism group}. 
In this section we recall the main objects under study and prove some preliminary results about them.
We refer the reader to~\cite{Freslon2023Book} for a comprehensive treatment of the theory of compact matrix quantum groups and quantum automorphism groups of graphs.

We start by introducing the key notion of a magic unitary.

Let $X$ be a unital $*$-algebra. A \emph{magic unitary} with coefficients in $X$ is a matrix $U = (u_{ij})_{i\in I, j\in J} $ where $I$ and $J$ are finite sets and such that for all $(i, j)\in I\times J$,
\begin{enumerate}
\item $u_{ij}^{2} = u_{ij} = u_{ij}^{*}$;
\item $u_{ij} u_{ik} = 0 $ for every $k\neq j$ and $ u_{ij} u_{kj} = 0 $ for every $k \neq i$; 
\item $\sum_{k \in I} u_{kj}  = 1 = \sum_{k \in J} u_{ik}$.
\end{enumerate}  
The definition implies that $\vert I\vert = \vert J\vert$ and that the matrix $U$ is a unitary in $\Mm_{I,J}(X)$. 
Notice that Condition (2) is automatic when $X$ is a $C^*$-algebra.

We can now define quantum permutation groups. 
\begin{definition}
	A \emph{quantum permutation group} is a pair $(A,U)$ where $A$ is a unital $*$-algebra generated by elements $U=(u_{ij})_{1\leq i,j\leq n} \in \Mm_n(A)$ for some integer $n \geq 1$ such that $U$ is a magic unitary and there exists a unital $*$-morphism $\Delta \colon A \to A \otimes A$ where $\Delta(u_{ij}) = \sum_{k=1}^{n} u_{ik} \otimes u_{kj}$ for all~$i,j$.   
\end{definition}

In the definition above, $\otimes$ is the unital algebraic tensor product. Moreover, the morphism $\Delta$ is unique and is called the \emph{comultiplication} of the quantum permutation group $(A,U)$. The matrix $U$ is called the \emph{fundamental representation} of the quantum permutation group. 

A quantum permutation group is a compact quantum group in the sense of Woronowicz~\cite{woronowicz1987compact,woronowicz1991remark}.

\begin{definition}
	Let $(A,U)$ and $(B,V)$ be two quantum permutation groups with comultiplications $\Delta_A$ and $\Delta_B$ respectively. 
	A \emph{morphism} of quantum groups $\varphi \colon (A,U) \to (B,V)$ is a unital $*$-morphism $\varphi \colon A\to B$ intertwining the comultiplications, that is $(\varphi\otimes \varphi)\Delta_A = \Delta_B\varphi$.
\end{definition}

We can now define the quantum automorphism groups of a graph as defined in~\cite{Banica2004qaut}.
\begin{definition}
	The \emph{quantum automorphism group} of a graph $G$, denoted by $\Qu(G)$, is the quantum permutation group $(A,U)$ where $A=A(G)$ is the universal unital $*$-algebra generated by elements $U=(u_{xy})_{x,y \in V(G)}$ satisfying:
	\begin{itemize}
		\item $U$ is a magic unitary;
		\item $U\Adj(G) = \Adj(G)U$.
	\end{itemize}
\end{definition}
It then follows that there exists a $*$-homomorphism $\Delta \colon A(G)\to A(G)\otimes A(G)$ such that
\begin{equation*}
	\Delta(u_{ij}) = \sum_{k=1}^{N}u_{ik}\otimes u_{kj}.
\end{equation*}

Given an integer $n$, we define \emph{the} quantum permutation group \emph{$S_n^+$} as the quantum automorphism group of the empty graph on $n$ vertices, which is coherent with the original definition by Wang~\cite{Wang1998}.

One can recover the classical automorphism group from $\Qu(G)$ in the following way: first, the abelianisation of $A(G)$ is the algebra $C(\Aut(G))$ of complex-valued functions on $\Aut(G)$. 
Second, if we denote by $\pi \colon A(G) \to C(\Aut(G))$ this quotient map, and if we identify $C(\Aut(G))\otimes C(\Aut(G))$ with $C(\Aut(G)\times \Aut(G))$, then for $x \in A(G) $ and $g,h \in \Aut(G)$, we have:
\begin{equation*}
	(\pi\otimes \pi)\circ\Delta(x)(g, h) = \pi(x)(gh).
\end{equation*}
That is to say, the abelianisation of $\Qu(G)$ is the Gelfand dual of $\Aut(G)$.

Finally, we recall the fundamental notion of quantum isomorphism of graphs.
\begin{definition}
	A \emph{quantum isomorphism} from a graph $G$ to a graph $H$ is a magic unitary $U$ with coefficients in a unital $C^*$-algebra indexed by $V(H)\times V(G)$ such that
	\begin{equation*}
		U\Adj(G) = \Adj(H)U.
	\end{equation*} 
\end{definition}

Let us recall the following result of Fulton (see~\cite[Sec 3.2]{FultonThesis}).
\begin{lemma}[Fulton, 2006]\label{lem:fulton}
Let $G$ be a graph and let $U$ be a magic unitary commuting with $\Adj(G)$. Let $x$, $y$, $a$, $b\in V(G)$. If $u_{xa}u_{yb}\neq 0$, then $d_{G}(x,y) = d_{G}(a,b)$. 
\end{lemma}

When decomposing graphs, stabilisers of the automorphism group, which are subgroups of the automorphism group, appear naturally. 
There is an analogous phenomenon for the quantum automorphism group which we introduce here. 

Let $G$ be a graph and let $R\subseteq V(G)$. 
We say a magic unitary $U$ \textit{stabilises $R$} if $u_{xy} = 0$ whenever $x\in R$ and $y\notin R$, or $x\notin R$ and $y\in R$. 
Equivalently, this means that $U = \Diag(U_0,U_1)$ is block-diagonal, with $U_1$ indexed by $R\times R$. 
This makes the following definition natural.

\begin{definition}\label{def:stabiliser}
	Let $G$ be a graph on $n\geq 1$ vertices and let $R\subseteq V(G)$. 
	The \textit{stabiliser} of $R$ in $\Qu(G)$ is the pair $\Stab_{\Qu(G)}(R) = (A,U)$ where $A$ is the universal unital $*$-algebra generated by $n^2$ generators $U = (u_{xy})_{x,y\in V(G)}$ satisfying the following relations:
	\begin{itemize}
		\item $U$ is a magic unitary,
		\item $U\Adj(G) = \Adj(G)U$,
		\item $U$ stabilises $R$.
	\end{itemize}
\end{definition}

This defines a quantum permutation group, which is clearly a quotient of the quantum automorphism group of the graph. 
The proof is essentially the same as the one proving that $\Qu(G)$ is a quantum permutation group and we omit it here.

When $R = \{x\}$, we sometimes write $\Stab_{\Qu(G)}(R) = \Fix_{\Qu(G)}(x)$.

\medskip 

A \emph{marked graph} is a pair $(G,R) $ where $ G $ is a graph and $ R\subseteq V(G) $. 
A \emph{quantum isomorphism of marked graphs} from $ (G, R) $ to $ (H, S) $ is a quantum isomorphism $ U $ from $ G $ to $ H $ that preserves the marks, that is $ u_ax = 0 $ if $ a \in S $ and $x\notin R$ or $ a \notin S $ and $ x\in R$. 
When $(G,R)=(H,S)$, we say that $U$ is \emph{adapted} to $(G,R)$. 
We define the \emph{quantum automorphism group} of $(G, R) $ to be $ \Qu(G,R) = \Stab_{\Qu(G)}(R) $. 
When $R = \{r\}$, we often write $(G,r)$ instead of $(G,\{r\})$ as a slight abuse of notation. 

While marked graphs are close to 2-coloured graphs, they behave differently with respect to quantum isomorphism, and even classical isomorphism. In \Cref{fig:marked_coloured}, the two underlying graphs are isomorphic as coloured graphs, but they are not isomorphic as marked graphs. 

\begin{figure}
	\begin{tikzpicture}
		\tikzset{bullet/.style={circle, fill=black, inner sep=1.5pt}}
		\tikzset{every node/.style={font=\small}}
		
		\node at (0,-1) {$(H, \{c\})$};
		
		\node[bullet, label=above:$a$] (a) at (0,1.5) {};
		\node[bullet, label=below:$b$] (b) at (-1,0) {};
		\node[bullet, label=below:$c$] (c) at (1,0) {};
		
		\draw (a) -- (b) -- (c) -- (a);
		
		\draw[dashed] (c) circle (0.4);
		
		\node at (6,-1) {$(G, \{x,y\})$};
		
		\node[bullet, label=above:$x$] (x) at (6,1.5) {};
		\node[bullet, label=below:$y$] (y) at (5,0) {};
		\node[bullet, label=below:$z$] (z) at (7,0) {};
		
		\draw (x) -- (y) -- (z) -- (x);
		
		\draw[dashed, , rotate around={-30:(5.5,0.75)}] (5.5,0.75) ellipse (0.5 and 1.4);
	\end{tikzpicture}
	\caption{Two graphs that are isomorphic as 2-coloured graphs but not as marked graphs.} \label{fig:marked_coloured}
\end{figure}
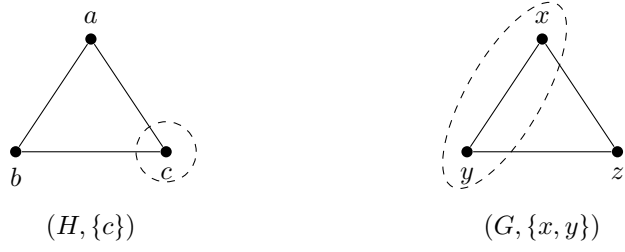

We will only use two specific types of marked graphs in this paper. A marked graph $ (G,R) $ is an \emph{anchored graph} if the intersection of $ R $ with every connected component of $ G $ is either a block or a cut vertex. 
Since we will only use connected anchored graphs, we will implicitly assume our anchored graphs to be connected in the rest of this paper. 
Finally, a marked graph $ (G,R) $ is called a \emph{rooted graph} if $R$ contains exactly one vertex per connected component of $G$.

Anchored graphs are practical to study the quantum automorphism group of a graph since graphs always come with a natural anchor. Indeed, by \Cref{thm:centre_block}, the centre $Z(G)$ of a connected graph $ G $ is either a cut vertex $ \{r\} $ or is contained in a unique block $ B_Z(G) $. We set $Q=\{r\}$ in the former case and $Q=B_Z(G)$ in the latter case. 
\begin{proposition} \label{prop:every_graph_anchored}
	Let $G$ be a connected graph, $ Q \subseteq V(G) $ as above, and $ U $ a magic unitary commuting with $\Adj(G)$. Then, $U$ is adapted to $(G,Q)$. In particular, there is a canonical isomorphism $\Qu(G) \simeq \Qu(G,Q)$. 
\end{proposition}
\begin{proof}
	This is a direct consequence of~\cite[Lemma 3.7]{freslon2025blockstructures}. 
\end{proof}

It is well-known that the quantum automorphism group of iterated copies of a graph decomposes as a wreath product with a quantum permutation group. See~\cite{bichon2004wreath} for the original formulation in a slightly different setting,~\cite[Theorem 6.1]{banica2007free}, or, for a presentation closer to our terminology,~\cite[Theorem 3.8]{Meunier2023}. 
Since we work with rooted graphs in this paper, let us present the corresponding formulation here.

\begin{lemma}\label{lem:qu_rooted_sums}
	Let $(G,r)$ be a connected rooted graph, and let $d\geq 1$. 
	We have $\Qu(d.(G,r)) = \Qu(G,r) \wr S_d^+$, with the natural actions.
\end{lemma}

\begin{proof}
	Let $\Qu(G,r) = (A,U)$, $\Qu(d.(G,r)) = (B,V)$, and $S_d^+ = (X,S)$. 
	Let $\Qq = \Qu(G,r)\wr S_d^+$, and write $\Qq = (C,Q)$. 
	For $1\leq i\leq d$ and $x\in V(G)$, we denote by $(G_i,r_i)$ the $i$th copy of $(G,r)$ in $d.(G,r)$, and by $x_i$ the copy of $x$ in $G_i$. 
	We let $R = \{r_1,\ldots,r_d\}$. 

	The proof follows closely the proof of Theorem 3.8 of~\cite{Meunier2023} by adapting it to the rooted case. 
	We introduce $W = (w_{xy})_{x,y\in V(dG)} \in \Mm_{V(dG)}(C)$ as follows: let $x$, $y\in V(d.G)$. 
	Let $i$ and $j$ be such that $x\in V(G_i)$ and $y\in V(G_j)$. 
	We define $w_{xy} = s_{ij}\nu_i(u_{xy})$. 
	In~\cite{Meunier2023}, it is shown that $W$ is a magic unitary commuting with $\Adj(d.G)$. 
	Let us show it stabilises the set of roots: assume that $x \notin R$. 
	Then $w_{xr_j} = s_{ij}\nu_i(u_{xr}) = 0$ since $u_{xr} = 0$. 
	Similarly $w_{r_jx} = 0$. 
	Hence $W$ stabilises the roots. 
	By universal property, this allows to obtain a (unique) unital $*$-morphism $\varphi \colon B \to C$. 
	The same computation than in~\cite{Meunier2023} shows that it commutes with the comultiplications. 

	Now let $1\leq i\leq d$. 
	We want to build a magic unitary $Z_i \in \Mm_{V(G)}(A)$. 
	For $x$, $y\in V(G)$, we define $z_{i,xy} = \sum_{j=1}^d v_{x_iy_j}$ and $Z_i = (z_{i,xy})_{x,y\in V(G)}$. 
	In~\cite{Meunier2023}, it is shown that $Z_i$ is a magic unitary commuting with $\Adj(G)$. 
	Let us check that it preserves the root. 
	Let $x\in V(G) \setminus\{r\}$. 
	We have $z_{i,xr} = \sum_{j=1}^d v_{x_i,r_j} = 0$ since $x_i\notin R$ and $V$ preserves $R$. 
	Similarly, we have $z_{i,rx} = 0$. 
	Hence $Z_i$ fixes $r$, as desired, and by universality we obtain a (unique) unital $*$-morphism $f_i\colon A\to B$. 

	The morphism $f_0 \colon X\to B$ is defined in the same way as in~\cite{Meunier2023} (it is also the morphism $\lambda_\Pp$ appearing in~\Cref{thm:snplus_qu_partitioned} with $\Pp$ the partition into connected components).

	By universal property, this gives us a morphism $f\colon C \to B$, and it is the inverse of $\varphi$ by the same proof as~\cite{Meunier2023}. 
	Since $\varphi$ is a morphism of quantum groups, we obtain the desired result.
\end{proof}

Let us fix some notations. 
For $(X_1,x_1),\ldots,(X_n,x_n)$ connected rooted graphs for some $n\geq 1$, we write $(X,R) = \bigoplus_{i=1}^n (X_i,x_i)$ when $X = \bigoplus_{i=1}^m X_i$ is the disjoint union of the $X_i$ and $R = \{x_1,\ldots,x_m\}$. 

Let $(G,R)$ be a rooted graph. 
Let $G = \bigoplus_{i=1}^k G_i$ be its decomposition into its connected components and let $\{ r_i \} = R\cap V(G_i)$ for every $1\leq i\leq k$. 
Let $I_G = \{1,\ldots,k\}$ and let $\Jj_G$ be the partition of $I_G$ into the equivalence classes of the relation $i\sim_G j$ if and only if $(G_i,r_i) \qi (G_j,r_j)$. 
We write $\Jj_G = \{J_1,\ldots,J_m\}$. 

\begin{proposition}\label{thm:qu_sums_rooted_graphs}
	We have
	\[ \Qu(G,R) = \bigast_{l=1}^m \Qu\left(\bigoplus_{i\in J_l} (G_i,r_i)\right)\]
	with the natural action. 
	Moreover, if $(G_i,r_i) \not \qi (G_j,r_j)$ for $i\neq j$, then, given $a_1,\ldots,a_k \geq 0$, we have
	\[ \Qu\left(\bigoplus_{i=1}^k a_i.(G_i,r_i)\right) = \bigast_{i=1}^k \Qu(G_i,r_i)\wr S_{a_i}^+\] with the natural action. 
\end{proposition}

\begin{proof}
	Let $\Qu(G,R) = (A,U)$. 
	Recall that, by~\Cref{lem:mu_partition_connected_components}, $U$ preserves the partition of $G$ into connected components. 
	Hence by~\Cref{thm:partitioned_graphs} if $u_{xy}\neq 0$ for some $x\in V(G_i)$ and $y\in V(G_j)$, then $U_{ij} = U[V(G_i),V(G_j)]$ is a quantum isomorphism from $G_j$ to $G_i$ with values in $p_{ij}Ap_{ij}$, where $p_{ij}$ is equal to the sum of the rows or columns of $U_{ij}$. 
	Since $U$ stabilises $R$, we have that $u_{r_ir_j} = p_{ij}$, so $U_{ij}$ is a quantum isomorphism of rooted graphs from $(G_j,r_j)$ to $(G_i,r_i)$. 
	Hence $U$ is diagonal by block, with every block indexed by $J_l\times J_l$ for some $1\leq l\leq m$. 
	This implies that $\Qu(G,R) = \bigast_{l=1}^m \Qu(\bigoplus_{i\in J_l} (G_i,r_i))$.

	Now, the second case corresponds to the case where if $i$,$j\in \Jj_l$ for some $1\leq l\leq m$, then $(G_i,r_i)\simeq(G_j,r_j)$. 
	Hence it follows from the first case and~\Cref{lem:qu_rooted_sums}. 
\end{proof}

\subsection{Superrigidity and tractability} \label{subsec:superrigid_tractable}

We define here the terms that are informally referred to as \emph{noncommutative properties} or \emph{quantum properties} of graphs.

Let $ G $ be a graph. 
Recall that the abelianisation of $\Qu(G)$ is the Gelfand dual $C(\Aut(G))$ of the classical automorphism group of $G$. 
Hence, we have that $\Qu(G) = C(\Aut(G))$ if the algebra of $\Qu(G)$ is commutative. 
We say that $G$ has \emph{quantum symmetry} if $ \Qu(G) \neq C(\Aut(G)) $, equivalently, if the underlying algebra of $\Qu(G)$ is not commutative. 
A graph $ G $ is \emph{asymmetric} if $\Aut(G)$ is trivial. 
Similarly, $G$ is \emph{quantum asymmetric} if $\Qu(G)$ is trivial. 
We say that $G$ satisfies \emph{Schmidt's criterion} if $ G $ admits two non-trivial automorphisms with disjoint supports. 
This criterion was introduced in~\cite{schmidt2020criterion} where it is shown that a graph satisfying it has quantum symmetry.

Now, let us define some noncommutative properties for families of graphs, following~\cite{Meunier2023}. 
Let $\Ff$ be a family of graphs, we say that $\Ff$ satisfies the axiom:
\begin{itemize}
	\item (QA) -- for \emph{quantum asymmetry} -- if a graph in $\Ff$ is quantum asymmetric if and only if it is asymmetric;
	\item (SA) -- for \emph{Schmidt's alternative} -- if a graph in $\Ff$ has quantum symmetry if and only if it satisfies Schmidt's criterion;
	\item (QI) -- for \emph{quantum isomorphism} -- if two graphs in $ \mathcal F $ are quantum isomorphic if and only if they are isomorphic;
	\item (SR) -- for \emph{supperrigidity} -- if whenever a graph is quantum isomorphic to a graph in $\Ff$, it is isomorphic to it. 
\end{itemize}

Note that (SR) is a strengthening of (QI). One can also show that (SA) implies (QA).

We say that a family $\Ff$ of graphs is \emph{superrigid tractable} family if (SA) and (SR) (and thus all the axioms) hold for $\Ff$.

Finally, one can naturally extend the definitions above to the setting of rooted graphs (or more generally, marked graphs), which we avoid recalling here. 
It is also straightforward that a rooted graph satisfying the rooted version of Schmidt's criterion has quantum symmetry.

\section{Quantum automorphism group of partitioned graphs}\label{sec:partitioned_graphs}

Partitions appear naturally in the study of the quantum automorphism groups of graphs, the typical example being the partition of a graph into its connected components. More generally, a useful technique for studying the noncommutative properties of a given graph $ G $ is to find an appropriate partitioned graph, typically based on some decomposition of $ G $, that has similar noncommutative properties to $ G $, and use the extra structure of the partitioned graph to study these properties. 
In~\cite{freslon2025blockstructures}, we defined partition graphs. In this section, we define and study their quantum automorphism groups.

A \emph{partitioned graph} is a pair $(G, \Pp)$ formed by a graph $G$ and a partition $\Pp$ of its vertex set. We denote by $\sim_{\Pp}$ the equivalence relation on $V(G)$ whose equivalence classes are the cells of $\Pp$. 
Let $ (G, \Pp) $ and $ (H, \Qq) $ be two partitioned graphs. A quantum isomorphism $ U $ from $ G $ to $ H $ is said to be a \emph{quantum isomorphism of partitioned graphs} if for $x, y \in V(G)$ and $a, b\in V(H)$ we have $ u_{ax}u_{by} = 0 $ if $x\sim_{\Pp} y$ and $a\not \sim_{\Qq} b$ or if $ x\not \sim_{\Pp} y$ and $a \sim_{\Qq} b $. 
In particular, when $ (H, \Qq) = (G, \Pp) $, we say that $ U $ is a magic unitary \emph{adapted to $(G,\Pp)$}. Notice that it means in particular that $ U $ commutes with the adjacency matrix of $ G $. This naturally yields to the following definition:

\begin{definition} \label{def:qaut_partitions}
	Given a partitioned graph $(G, \Pp)$, we define $\Qu(G, \Pp) = (A, U)$ where $ A $ is the universal unital $*$-algebra generated by elements $U = (u_{ij})_{i, j\in V(G)}$ such that
	\begin{enumerate}
		\item The matrix $U $ is a magic unitary;
		\item $U$ is adpated to $(G, \Pp)$.
	\end{enumerate}
\end{definition}

Note that this can also be defined as the quotient of $\Qu(G)$ by the relations making $U$ adapted to $(G, \Pp)$. 

\begin{remark}
	The setting of partitioned graphs is useful because of its generality. Indeed, observe that we do not require $U$ to satisfy $u_{xy} = 0$ when $x$ and $y$ belong to different cells of the partition $\Pp$, so that this is more general, for instance, than quantum isomorphisms of coloured graphs.
\end{remark}

Let us check that the definition above indeed defines a compact quantum group.

\begin{lemma}\label{lem:qu_partitioned_graphs}
	The pair $\Qu(G, \Pp) = (A, U)$ is a quantum permutation group. In other words, there exists a unital $*$-homomorphism $\Delta \colon A \to A\otimes A$ such that for all $x$, $y\in V(G)$
	\begin{equation*}
		\Delta(u_{xy}) = \sum_{z\in V(G)} u_{xz}\otimes u_{zy}.
	\end{equation*}
\end{lemma}
\begin{proof}
	Let $B = A\otimes A$ be the unital algebraic tensor product of $*$-algebras of $A$ with itself, and let $V = (v_{xy})_{x,y\in V(G)} \in \Mm_{V(G)}(B)$ be defined by $v_{xy} = \sum_{z\in V(G)} u_{xz}\otimes u_{zy}$. We prove that $ V $ satisfies the properties of \Cref{def:qaut_partitions}. The first item, as well as the fact that $ V $ and $ \Adj(G) $ commute are straightforward computations. Assuming that $x\sim_{\Pp} y$ and $a\not \sim_{\Pp} b$, we have
	\begin{equation*}
		v_{xa}v_{yb} = \sum_{s,t\in V(G)} u_{xs}u_{yt}\otimes u_{sa}u_{tb}.
	\end{equation*}
	Now if $s\sim_{\Pp} t$, we have $u_{sa}u_{tb} = 0$, and if $s\not \sim_{\Pp} t$, then $u_{xs}u_{yt} = 0$. This shows that $v_{xa}v_{yb} = 0$. One proves similarly that $v_{ax}v_{by} = 0$. By universality of $A$, there exists a (unique) unital $*$-morphism $\Delta \colon A\to A\otimes A$ such that $\Delta(u_{xy}) = v_{xy}$ for any $x$, $y\in V(G)$, which concludes the proof.
\end{proof}

If $\Pp$ consists of the whole graph, then clearly $\Qu(G, \Pp) = \Qu(G)$. There is a more interesting instance of that equality which we record now for later use.

\begin{lemma}\label{lem:mu_partition_connected_components}
	Let $G$ be a graph, let $\Pp$ be the partition of $V(G)$ into the vertex sets of its connected components, and let $U$ be a magic unitary. If $U$ commutes with $\Adj(G)$, it preserves $\Pp$. In particular, there is a canonical isomorphism $\Qu(G) \simeq \Qu(G,\Pp)$.
\end{lemma}

\begin{proof}
	This is a direct consequence of~\cite[Lemma 3.4]{freslon2025blockstructures}.
\end{proof}

Let $(G,\Pp_{G})$ and $(H,\Pp_{H})$ be two partitioned graphs, and let $U$ be a quantum isomorphism from $(G,\Pp_G)$ to $(H,\Pp_H)$. We define the matrix $P(U) = (p_{CD})_{C\in \Pp_{G}, D\in \Pp_{H}}$ as follows (see~\cite[Sec 3]{freslon2025blockstructures} for more details). For $C\in \Pp_{H}$ and $D\in \Pp_{G}$, we have for any $ a \in C$ and $ x \in D$
\begin{equation*}
	\sum_{ c \in C} u_{cx} = p_{CD} = \sum_{d \in D} u_{ad}.
\end{equation*}

The following fundamental result was obtained in~\cite[Theorem 3.3]{freslon2025blockstructures}. 

\begin{theorem}\label{thm:partitioned_graphs}
	Let $(G, \Pp_{G})$ and $(H, \Pp_{H})$ be two partitioned graphs, and let $U$ be a quantum isomorphism of partitioned graphs from $(G, \Pp_{G})$ to $(H, \Pp_{H})$ with coefficients in a unital $C^{*}$-algebra $X$. Let moreover $C\in \Pp_{H}$ and $D\in \Pp_{G}$ and set $W = W_{C, D} = U[C,D]$ and $p = p_{CD}$. 
	Then,
	\begin{enumerate}
		\item The matrix $P = (p_{KL})_{K\in \Pp_{H}, L\in \Pp_{G}}$ is a magic unitary with coefficients in $X$;
		\item If $p\neq 0$, then it is the unit of the $C^{*}$-subalgebra of $X$ generated by the coefficients of $W$;
		\item If $p = 0$, then $W = 0$;
		\item The matrix $W$ is a magic unitary with coefficients in the $C^{*}$-algebra $pXp$;
		\item $\Adj(H)[C,C]W = W\Adj(G)[D,D]$.
	\end{enumerate}
	In particular, there is a bijection $\varphi \colon \Pp_{G} \to \Pp_{H}$ such that $W_{D, \varphi(D)}$ is a quantum isomorphism from $G[D]$ to $H[\varphi(D)]$ for every $D\in \Pp_{G}$.
\end{theorem}

A graph automorphism respecting a given partition $\Pp$ induces a permutation between the cells of $\Pp$, thereby providing a group homomorphism $\Aut(G, \Pp)\to S_{\vert\Pp\vert}$. We will now prove a quantum analogue of this.

\begin{theorem}\label{thm:snplus_qu_partitioned}
	Let $(G,\Pp)$ be a partitioned graph, and let $\Pp = \{C_{1},\ldots,C_{k}\}$. Let $\Qu(G,\Pp) = (A,U)$ and $S_{k}^{+} = (B, S)$. Then, there exists a unique quantum group morphism $\lambda_{\Pp} \colon S_{k}^{+} \to \Qu(G, \Pp)$ such that $\lambda_{\Pp}(s_{ij}) = [P(U)]_{ij}$. 
\end{theorem}

\begin{proof}
	By~\Cref{thm:partitioned_graphs}, we have that $P(U) = (p_{ij})_{1\leq i,j\leq k}$ is a well-defined magic unitary, with $p_{ij}$ being the sum of rows or columns of the block $U[C_{i}, C_{j}]$. Hence, by the universal property of $B$, there exists a unique unital $*$-homomorphism $\lambda_{\Pp} \colon B\to A$ such that $\lambda_{\Pp}(s_{ij}) = p_{ij}$ for $1\leq i,j\leq k$. All we have to check is therefore that $\lambda_{\Pp}$ intertwines the comultiplications.
	
	Let $\Delta_{A}$ and $\Delta$ denote the comultiplications of $\Qu(G, \Pp)$ and $S_{k}^{+}$ respectively. If we fix for each $1\leq l\leq k$ an element $w_{l} \in C_{l}$, then for any $1\leq i,j\leq k$,
	\begin{align*}
		\Delta_{A}\circ\lambda_{\Pp}(s_{ij}) & = \Delta_{A}(p_{ij}) = \sum_{x\in C_{i}} \Delta_{A}(u_{xw_{j}}) = \sum_{z\in V(G)} \sum_{x\in C_{i}} u_{xz}\otimes u_{zw_{j}} \\
		& = \sum_{l=1}^{k} \sum_{z\in C_{l}} \sum_{x\in C_{i}} u_{xz}\otimes u_{zw_{j}} = \sum_{l=1}^{k} \sum_{z\in C_{l}} p_{il} \otimes u_{zw_{j}} = \sum_{l=1}^{k} p_{il}\otimes p_{lj} \\
		& = \sum_{l=1}^{k+1} \lambda_{\Pp}(s_{il})\otimes \lambda_{\Pp}(s_{lj}) = (\lambda_{\Pp}\otimes \lambda_{\Pp})\circ \Delta(s_{ij}),
	\end{align*}
	as claimed.
\end{proof}

\section{The $\Psi$-operation}\label{sec:psi}

Since the quantum automorphism group of a graph is the quantum automorphism group of a corresponding anchored graph by \Cref{prop:every_graph_anchored}, we now define a transformation $\Psi$ that given an anchored graph returns a rooted graph to obtain a decomposition of the quantum automorphism group of a graph. This operation is very close to the $\Gamma$-operation from~\cite[Section 4.3]{freslon2025blockstructures}, in particular, the underlying graphs of $ \Psi(G,Q) $ and $ \Gamma(G, Q) $ are isomorphic. However, $\Gamma$ returns an anchored graph which is more adapted to the context of block structures, and $\Psi$ returns a rooted graph which is the correct tool for our decomposition here.

Let us define $\Psi$. Let $ (G,Q)$ be an anchored graph. 
\begin{itemize}
	\item If $Q = \{r\}$ is a cut vertex, let $C_{1}, \ldots, C_{k}$ be the connected components of $G\setminus \{r\}$ (notice that $k\geq 2$). 
		For $1\leq i\leq k$, we set $G_{i} = G[V(C_{i})\cup\{r\}]$, that is, every $G_{i}$ is the graph obtained by adding back $r$ to $C_{i}$. 
		Let now
		\begin{equation*}
			G' = \bigoplus_{i=1}^{k} G_{i}
		\end{equation*}
		and note that for $1\leq i\leq k$, $G_{i}$ is exactly one connected component of $G'$, and that it contains a copy of $r$, that we will denote by $r_{i}$. 
		Setting $R = \{r_{1}, \ldots, r_{k}\}$ yields a rooted graph $(G', R) = \Psi(G,Q)$.
	\item Otherwise, $Q$ is a block of $G$. 
		Define $G'$ to be the graph obtained from $G$ by removing all edges in $Q$, that is, $V(G')= V(G)$ and $E(G')=E(G) \setminus \{ uv \mid u,v \in Q \}$. 
		Then $(G', Q) = \Psi(G,Q)$ is again a rooted graph.
\end{itemize}

The point of that construction is that it carries on to quantum isomorphisms. 
Let $(G,Q_G)$ and $(H,Q_H)$ be two connected anchored graphs, and let $U$ be a quantum isomorphism of anchored graphs from $(G,Q_G)$ to $(H,Q_H)$. Notice that the fact that the size of the anchor is preserved implies that $ Q_H $ is a cut vertex if and only if $ Q_G $ is so. 
Let us define a matrix $\Psi(U)$ in the following way ($\Psi(U)$ should be seen merely as a notation since strictly speaking, $\Psi$ here is not a function):
\begin{itemize}
\item If $Q_G = \{r\}$ (and hence $Q_H = \{s\}$), then $U = \Diag(U_{0}, 1)$, where $1= u_{sr}$, and we set $\Psi(U) = (U_{0}, P(U_{0}))$, where $P(U_{0})$ comes from the partition of the graph $G\setminus\{r\}$ into connected components.
\item If $Q_G$ is a block (hence $Q_H$ is a block of $H$), then we set $\Psi(U) = U$.
\end{itemize}

\begin{lemma}\label{lem:psi_u}
	Let $(G,Q_G)$ and $(H,Q_H)$ be two connected anchored graphs, and let $U$ be a quantum isomorphism from $(G,Q_G)$ to $(H,Q_H)$. 
	Then, $\Psi(U)$ is a quantum isomorphism of rooted graphs from $\Psi(G,Q_G)$ to $\Psi(H,Q_H)$.
\end{lemma}

\begin{proof}
	By~\cite[Lemma 4.8]{freslon2025blockstructures}, since $\Psi(U)=\Gamma(U)$, we have that $\Psi(U)$ is a quantum isomorphism so that we simply have to check that the roots are preserved. 
	If $Q_G$ is a cut vertex, this follows directly from the construction. In the other case, this follows from $\Psi(U) = U$ and the fact that $U = \Diag(U_{0}, U_{1})$ with $U_{1}$ indexed by $Q_H\times Q_G$.
\end{proof}

It is clear that the connected components of $\Psi(G,Q)$ all have strictly less vertices than $G$ itself, so that this decomposition can be used for inductive proofs. In particular, using it for the computation of the quantum automorphism groups of families of graphs requires relating $\Qu(\Psi(G,Q))$ to $\Qu(G,Q)$. Let us now establish this relationship.

\begin{theorem}\label{thm:quotient}
	Let $(G, Q)$ be a connected anchored graph. 
	Let $\Qu(G, Q) = (A,U)$, $\Qu(\Psi(G,Q)) = (B,W)$, and $V = \Psi(U)$. 
	Then there exists a unique surjective quantum group morphism $\varphi\colon \Qu(\Psi(G,Q)) \to \Qu(G,Q) $ satisfying $\varphi(w_{xy}) = v_{xy}$.
\end{theorem}

\begin{proof}
	Let $\Psi(G,Q) = (G',R)$. 
	By~\Cref{lem:psi_u}, the matrix $V$ is a magic unitary commuting with $\Adj(G')$ and stabilising $R$. 
	Therefore, by universality, there exists a unique unital $*$-homomorphism $\varphi\colon A\to B$ satisfying $\varphi(w_{xy}) = v_{xy}$ for any $x$, $y\in V(G')$. 

	Let us check that $\varphi$ is surjective. 
	If $Q = \{r\}$ with $r$ a cut vertex of $G$, then $U = \Diag(U_0,1)$, and since the coefficients of $U$ generate $A$ we have that the coefficients of $U_0$ generate $A$. 
	Moreover, $U_0$ is indexed by $V(G)\setminus \{r\} = V(G')\setminus R$, hence $$[U_0]_{xy} = u_{xy} = v_{xy} = \varphi(w_{xy})$$ for any $x$, $y\in V(G')\setminus R$, proving surjectivity. 
	In the second case, we have $V = U$ and $V(G) = V(G')$, so $\varphi$ is clearly surjective. 

	Let us now show that $\varphi$ is a quantum group morphism, i.e. it intertwines the comultiplications. 
	In the case where $ Q $ is a block of $G$, this follows from the fact that $\varphi(w_{xy}) = u_{xy}$ for any $x$, $y\in V(G)$.

	We therefore now assume that $Q = \{r\}$ with $r$ a cut vertex of $G$. We denote by $\Delta_A$ and $\Delta_B$ the comultiplications of $\Qu(G,Q)$ and $\Qu(\Psi(G,Q))$ respectively. Recall that by construction $V = \Diag(U_0,U_1)$, where $U_0$ is indexed by $V(G)\setminus\{r\}$, and $U_1$ is indexed by $R = \{r_1,\ldots,r_k\}$. 
	By definition, $U_1 = P(U_0)$ is the permutation matrix obtained from $U_0$ associated to the partition of $G\setminus \{r\}$ into its connected components. 
	Moreover, we have $W = \Diag(W_0,W_1)$, where $W_1$ is indexed by $R$. 
	Let $x$, $y\in V(G')$. 
	
	If $x\in R$ and $y\notin R$, or if $x\notin R$ and $y\in R$, then $w_{xy} = 0$, hence by linearity 
	$$\Delta_A\varphi(w_{xy}) = 0 = (\varphi\otimes\varphi)\Delta_B(w_{xy}),$$ 
	as desired. 
	If $x,y\notin R$, then
	\begin{align*}
		(\varphi\otimes\varphi)(\Delta_B(w_{xy})) &= \sum_{z\in V(G')} \varphi(w_{xz})\otimes\varphi(w_{zy}) = \sum_{z\in V(G')\setminus R} v_{xz}\otimes v_{zy}\\
		&= \sum_{z\in V(G)\setminus\{r\}} u_{xz}\otimes u_{zy} = \Delta_A(u_{xy})\\
		&= \Delta_A(\varphi(w_{xy})),
	\end{align*}
	as desired. 
	Finally, if $x, y\in R$, set $\Pp = \{C_1,\ldots,C_k,\{r\}\}$, where the $C_l$'s are the connected components of $G\setminus \{r\}$. 
	By~\Cref{thm:snplus_qu_partitioned}, there exists a unique $*$-homomorphism 
	$\lambda_\Pp \colon S_{k+1}^+ \to \Qu(G, \Pp)$. Moreover, since $U$ fixes $r$, by~\Cref{lem:mu_partition_connected_components}, it preserves the partition~$\Pp$. 
	Hence, there exists a quantum group morphism $\lambda_Q \colon \Qu(G,\Pp) \to \Qu(G,Q)$ mapping the corresponding fundamental representations. Set 
	$$ \lambda = \lambda_Q \circ \lambda_{\Pp} \colon S_{k+1}^+ \to \Qu(G,Q). $$
	We have $\lambda(s_{ij}) = [P(U)]_{ij}$. 
	In particular, denoting by $\Delta_S$ the comultiplication of $S_{k+1}^+$, we have that 
	$$(\lambda \otimes\lambda)\Delta_S = \Delta_A\lambda.$$ 
	Moreover, since $U$ fixes $r$, we have $P(U) = \Diag(P(U_0),1)$. 
	Writing $P(U_0) = (p_{ij})_{1\leq i,j\leq k}$, we have that 
	for $1\leq i,j\leq k$,
	\begin{align*}
		(\varphi\otimes\varphi)(\Delta_B(w_{r_ir_j})) &= \sum_{l=1}^k \varphi(w_{r_ir_l})\otimes\varphi(w_{r_lr_j}) = \sum_{l=1}^k p_{il}\otimes p_{lj}\\
		&= \sum_{l=1}^k \lambda(s_{il})\otimes\lambda(s_{lj}) = (\lambda\otimes\lambda)(\Delta_S(s_{ij}))\\
		&= \Delta_A\lambda(s_{ij}) = \Delta_A(p_{ij}) = \Delta_A(\varphi(w_{r_ir_j})),
	\end{align*}
	as desired. 
	This concludes the proof. 
\end{proof}

The previous theorem gives a surjective morphism, but injectivity does not hold in general as the simple example in \Cref{fig:morphism_not_injective} shows. Let $ G = C_4 $ be the cycle on 4 vertices and let $Q=V(G)$. Then, $\Psi(G, Q) = (H, Q)$, where $ H $ is the empty graph on the vertex set of $ G $. We have $\Qu(G,Q) = S_4^+ \neq H_2^+ = \Qu(H,Q) $, where the last equality is from~\cite{Banica2009fusion} (see also~\cite[Exercise 7.4]{Freslon2023Book}). 
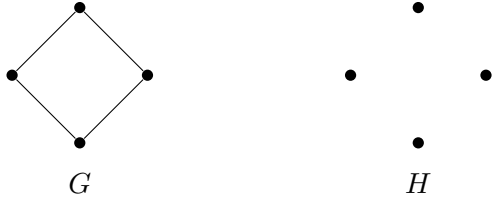
\begin{figure}
	\begin{tikzpicture}[scale=0.9]
		\tikzset{bullet/.style={circle, fill=black, inner sep=1.5pt}}
		\node at (1,-.6) {$G$};
		\node[bullet] (g1) at (0,1) {};
		\node[bullet] (g2) at (1,2) {};
		\node[bullet] (g3) at (2,1) {};
		\node[bullet] (g4) at (1,0) {};
		\draw (g1) -- (g2) -- (g3) -- (g4) -- (g1);
		
		\node at (6,-.6) {$H$};
		\node[bullet] (h1) at (5,1) {};
		\node[bullet] (h2) at (6,2) {};
		\node[bullet] (h3) at (7,1) {};
		\node[bullet] (h4) at (6,0) {};
	\end{tikzpicture}
	\caption{The morphism in \Cref{thm:quotient} is not an injection for $ G =C_4 $.} \label{fig:morphism_not_injective}
\end{figure}

However, this morphism tunes out to be an isomorphism when the anchor is complete.

\begin{theorem}\label{thm:iso}
	The map $\varphi$ in the statement of \Cref{thm:quotient} is an isomorphism if $Q$ is a complete graph.
\end{theorem}
\begin{proof}
	Let us set $\Qu(G,Q) = (A,U)$, $\Qu(\Psi(G,Q)) = (B,W)$, and $\Psi(G,Q) = (G',R)$. It is enough to build an inverse for $\varphi$, since an inverse automatically intertwines the comultiplications.
	
	We start with the case where $Q$ is a block. Then, $V(G) = V(G')$ and $W = \Diag(W_{0}, W_{1})$, with $W_{1}$ indexed by $R = Q$. We claim that $W$ commutes with $\Adj(G)$. To prove this, we set
	\begin{equation*}
		M = \Adj(G) = \begin{pmatrix}
			M_{0} &{}^{t}L\\
			L & M_{1}
		\end{pmatrix}
		\quad \text{ and } \quad 
		M' = \Adj(G') = \begin{pmatrix}
			M_{0} & {}^{t}L\\
			L & 0
		\end{pmatrix}
	\end{equation*}
	with $M_{1}$ indexed by $Q$. Since $M'W = WM'$ by definition of $W$, we have $W_{0}M_{0} = M_{0}W_{0}$, $W_{0}{}^{t}L = {}^{t}LW_{1}$, and $W_{1}L = LW_{0}$. Therefore, $M$ commutes with $W$ if and only if $W_{1}$ commutes with $M_{1}$. But $M_{1} = \Adj(G[Q])$, and $Q$ induces a complete graph. Because every magic unitary commutes with the adjacency matrix of a complete graph, we obtain that $W_{1}M_{1} = M_{1}W_{1}$, and thus $WM = MW$, as claimed. 
	
	Since $R=Q$, we have shown that $W$ is adapted to $(G,Q)$. By universality, this implies that there exists a (unique) unital $*$-homomorphism $\psi \colon B\to A$ such that $\psi(u_{xy}) = w_{xy}$ for any $x$, $y\in V(G)$. It is clear that on the generators $\psi$ is the inverse of $\varphi$, hence the result.
	
	We assume now that $Q = \{r\}$, with $r$ a cut vertex in $G$. In this case, we have
	\begin{equation*}
		\Adj(G) = M = 
		\begin{pmatrix}
			M_{0} & {}^{t}L\\
			L & 0
		\end{pmatrix}
		\quad \text{ and } \quad
		\Adj(G') = M' = 
		\begin{pmatrix}
			M_{0} & {}^{t}D_{L}\\
			D_{L} & 0
		\end{pmatrix},
	\end{equation*}
	where $L = (L_{1} \ldots L_{k})$ is a row matrix with $L_{i}$ indexed by $\{r\}\times C_{i}$ and $D_{L} = \Diag(L_{1}, \ldots, L_{k})$. Writing $W = \Diag(W_{0}, W_{1})$, we define a new matrix $T = \Diag(W_{0}, 1)$ indexed by $V(G)$, where $W_{1}$ is replaced by a single coefficient $1$, corresponding to $r$. We know that $W$ commutes with $M'$, so that $W_{0}M_{0} = M_{0}W_{0}$, and we want to show that $T$ commutes with $M$.
	
	Let us first show that $LW_{0} = L$. To do this, let $x\in V(G)\setminus \{r\}$, and let $1\leq j\leq k$ be the integer such that $x\in C_{j}$. Since $W$ commutes with $M'$, we have $D_{L}W_{0} = W_{1}D_{L}$. For $1\leq i\leq k$ and $y\in V(G)\setminus\{r\}$, note that $[D_{L}]_{r_{i}y} = 0$ if $y\notin C_{i}$ and $[D_{L}]_{r_{i}y} = l_{ry}$ if $y\in C_{i}$. In particular, $l_{ry} = \sum_{i=1}^{k} [D_{L}]_{r_{i}y}$, and we can therefore compute
	\begin{align*}
		[LW_{0}]_{rx} & = \sum_{z\in V(G)\setminus \{r\}} l_{rz}w_{zx} = \sum_{i=1}^{k} \sum_{z\in V(G)\setminus \{r\}} [D_{L}]_{r_{i}z} w_{zx} \\
		& = \sum_{i=1}^{k} [D_{L}W_{0}]_{r_{i}x} = \sum_{i=1}^{k} [W_{1}D_{L}]_{r_{i}x} \\
		& = \left(\sum_{i=1}^{k} w_{r_{i}r_{j}}\right) [D_{L}]_{r_{j}x} = l_{rx}.
	\end{align*}
	We have shown that $LW_{0} = L$, from which ${}^{t}W_{0}{}^{t}L = {}^{t}L$ follows. Moreover, $W_{0}$ being unitary, multiplying by it on the left yields ${}^{t}L = W_{0}{}^{t}L$. In conclusion, $T=\Diag(W_0,1)$ commutes with $M$.
	
	Again, $T$ is adapted to $(G,Q)=(G,r)$. By universality, the previous commutation relation implies the existence of a unique unital $*$-homomorphism $\psi \colon A\to B$ such that $\psi(u_{xy}) = t_{xy}$ for $x$, $y\in V(G)$. 
	
	Let us show that $\varphi$ and $\psi$ are inverse to one another. We start by showing that $\varphi\circ\psi = \id_{A}$. Since $U = \Diag(U_{0}, 1)$, the coefficients of $U_{0}$ generate $A$, and it is enough to show that $\varphi\circ \psi(u_{xy}) = u_{xy}$ for $x$, $y\in V(G)\setminus \{r\}$. But this follows immediately from the fact that $t_{xy} = w_{xy}$ and $u_{xy} = v_{xy}$, since $V(G')\setminus R = V(G)\setminus\{r\}$. Secondly, let us now show that $\psi\circ\varphi = \id_{B}$. Since $W = \Diag(W_{0}, W_{1})$, the coefficients of $W_{0}$ and $W_{1}$ generate $B$, so that it is enough to show that $\psi\circ\varphi$ is the identity on these coefficients. If $x \notin R$ and $y\notin R$, then $\psi\circ\varphi(w_{xy}) = \psi(u_{xy}) = t_{xy} = w_{xy}$ as above. Consider now $1\leq i,j\leq k$, and let us prove that $\psi\circ\varphi(w_{r_{i}r_{j}}) = w_{r_{i}r_{j}}$. Keeping the notations previously used in the proof, and using the equality $D_{L}W_{0} = W_{1}D_{L}$, we have for $x\in C_{j}$,
	\begin{equation*}
		\sum_{y\in C_{i}} [D_{L}]_{r_{i}y} w_{yx} = [D_{L}W_{0}]_{r_{i}x} = [W_{1}D_{L}]_{r_{i}x} = \sum_{l=1}^{k} w_{r_{i}r_{l}}[D_{L}]_{r_{l}x} = w_{r_{i}r_{j}}[D_{L}]_{r_{j}x}.
	\end{equation*}
	As a result of this formula, we have the following:
	\begin{itemize}
		\item If $x$ is a neighbour of $r_{j}$, then $[D_L]_{r_jx} = 1$ and therefore $w_{r_ir_j} = \sum_{y\in C_i} [D_L]_{r_iy}w_{yx}$;
		\item If $x$ is a neighbour of $r_{j}$ and $y\in C_{i}$ is such that $w_{yx}\neq 0$, then $w_{yx} = \sum_{a\in V(G')} w_{yx}w_{ar_{j}}$, so that there exists $a\in V(G')$ satisfying $w_{yx}w_{ar_{j}}\neq 0$.  By~\Cref{lem:fulton}, this implies $d(y,a) = d(x,r_{j}) = 1$, hence $y$ is a neighbor of $a$. Moreover, since $w_{ar_{j}}\neq 0$, $a\in R$, implying $a = r_{i}$. We have consequently shown that $y$ is a neighbor of $r_{i}$.
	\end{itemize} 
	Gathering the previous facts, taking $ x $ a neighbour of $r_j$, we have $w_{r_ir_j} = \sum_{y\in C_i} w_{yx}$ and
	\begin{equation*}
		\psi(\varphi(w_{r_{i}r_{j}})) = \psi(p_{ij}) = \sum_{y\in C_{i}} \psi(u_{yx}) = \sum_{y\in C_{i}} w_{yx} = w_{r_{i}r_{j}},
	\end{equation*}
	as desired. This shows that $\psi$ is an inverse to $\varphi$. This completes the proof.
\end{proof}

Using the results of~\Cref{subsec:QPG}, and in particular~\Cref{thm:qu_sums_rooted_graphs}, we obtain our main decomposition of $\Qu(G,Q)$, for a connected anchored graph $(G,Q)$. 
This decomposition is particularly useful when the graph $G$ is not 2-connected. Moreover, when~\Cref{thm:iso} applies, \Cref{thm:decomposition_qu_group} is enough to precisely compute the quantum automorphism group of~$G$. 

Let $(G',R) = \Psi(G,Q)$, and let $I = I_{G'}$, and $\Jj = \Jj_{G'}$  as in~\Cref{subsec:QPG}, before \Cref{thm:qu_sums_rooted_graphs}.

\begin{corollary}\label{thm:decomposition_qu_group_anchored}
	Let $(G,Q)$ be a connected anchored graph. There exists a surjective quantum group morphism
	\begin{equation*}
		\varphi\colon \bigast_{l=1}^{m} \Qu(\bigoplus_{j\in J_l} (G_i,r_i)) \to \Qu(G,Q),
	\end{equation*}
	which is an isomorphism when $Q$ is complete. Moreover, if $(G_{i}, r_{i})\simeq (G_{j}, r_{j})$ whenever $(G_{i}, r_{i})\qi(G_{j}, r_{j})$, then
	\begin{equation*}
		\varphi\colon \bigast_{l=1}^{m} \Qu(G_{i}, r_{i}) \wr_{\ast} S_{|J_i|}^{+} \to \Qu(G,Q).
	\end{equation*}
\end{corollary}

\begin{proof}
	This follows by combining~\Cref{thm:quotient},~\Cref{thm:iso}, and~\Cref{thm:qu_sums_rooted_graphs}.
\end{proof}

\begin{corollary}\label{thm:decomposition_qu_group}
	Let $G$ be a connected graph. Let $ Q\subseteq V(G)$ be defined as before \Cref{prop:every_graph_anchored}. Then, there exists a surjective quantum group morphism
	\begin{equation*}
		\varphi\colon \bigast_{l=1}^{m} \Qu(\bigoplus_{j\in J_l} (G_i,r_i)) \to \Qu(G),
	\end{equation*}
	which is an isomorphism when $Q$ is complete. Moreover, if $(G_{i}, r_{i})\simeq (G_{j}, r_{j})$ whenever $(G_{i}, r_{i})\qi(G_{j}, r_{j})$, then
	\begin{equation*}
		\varphi\colon \bigast_{l=1}^{m} \Qu(G_{i}, r_{i}) \wr_{\ast} S_{|J_i|}^{+} \to \Qu(G).
	\end{equation*}
\end{corollary}
\begin{proof}
	It follows directly from \Cref{thm:decomposition_qu_group_anchored} and \Cref{prop:every_graph_anchored}. 
\end{proof}

Indeed, the precedent corollary remains true for any anchor $Q$ stabilised by the fundamental representation of $\Qu(G)$.

\section{Complete blocks}\label{sec:complete_blocks}

Before studying the noncommutative properties of block graphs, we show a key graph-theoretical result,~\Cref{thm:centre_block}, characterising the centre of a graph satisfying the hypotheses that the centre is contained in a complete block (i.e.\ a block that induces a complete subgraph).  More precisely, we show that the centre of the graph is contained in a complete block if and only if it is the whole complete block. This allows us to apply \Cref{thm:iso}. 

Let us start by a lemma comparing the eccentricity between cut vertices and the other vertices of a complete block.

\begin{lemma}\label{lem:compare_e_internal_cut}
	Let $G$ be a graph, let $B$ be a complete block of $G$, and let $u, v \in B$. 
	If $v$ is a cut vertex and $u$ is not a cut vertex, then $e(u) - 1 \leq e(v) \leq e(u)$.
\end{lemma}

\begin{proof}
	First, let $x$ be an eccentric vertex for $u$ and let $w$ be the first vertex in $B$ of an $x-B$ path. 
	We claim that 
	\begin{equation} \label{eq:case1}
		d(v,x) = \begin{cases}
			d(u,x) - 1 & \text{ if } w = v \\
			d(u,x) & \text{ if } w \neq v 
		\end{cases}.
	\end{equation}
	Let us prove this claim. 
	If $x \in B$, we have $w=x$ and the claim is obvious. 
	Assume therefore that $x \notin B$. 
	Then, every path between a vertex in $B$ and $x$ passes through $w$. 
	In particular, a shortest path from $u$ to $x$ passes through $w$, which implies that $d(u,x) = d(u,w) + d(w,x) = d(w,x) + 1$. 
	Now, if $w = v$, we have $d(u,x) = d(v,x) + 1$ and if $w \neq v$, then every shortest path between $v$ and $x$ starts with $vw$ and continues with a shortest path from $w$ to $x$, hence $d(v,x) = 1 + d(w, x) = d(u,x)$. 
	This completes the proof of~\Cref{eq:case1}. 
	
	Now, let $y$ be an eccentric vertex for $v$, and let $w$ be the first vertex in $B$ of a $y - B$ path. 
	We claim that 
	\begin{equation} \label{eq:case2}
		d(u,y) = \begin{cases}
			d(v,y) + 1 & \text{ if } w = v \\
			d(v,y) & \text{ if } w \neq v
		\end{cases}.
	\end{equation}
	To prove the claim, note that if $y\in B$, we have $w=y$ and the claim follows, and when $y \notin B$, a shortest path between $y$ and any vertex $b \in B$ passes through $w$, thus $d(b, y) = d(b,w) + d(w,y)$. 
	Now, $u$ is not a cut vertex and $w \in B$, so $d(u,w) =1$. 
	Since $u\in B$, we have $d(u,y)  = d(u,w) + d(w,y) = 1 + d(w,y)$. 
	Moreover, we have $d(v,y) = d(v,w) + d(w,y)$, and $d(v,w)$ is 1 if $v \neq w$ and is 0 if $v = w$. 
	This completes the proof of~\Cref{eq:case2}. 
	
	Finally, from~\Cref{eq:case1}, we have 
	$$e(u)-1  = d(u,x)-1 \leq d(v,x) \leq e(v),$$
	and from~\Cref{eq:case2}, we have 
	$$e(v) = d(v,y) \leq d(u,y) \leq e(u),$$ 
	which completes the proof.  
\end{proof}

Next, we compare the eccentricity of cut vertices belonging to the same block. 

\begin{lemma} \label{lem:compare_e_cut_cut}
	Let $G$ be a graph and let $v$ and $v'$ be two distinct cut vertices in a complete block $B$ of $G$. 
	Let $x$ be an eccentric vertex for $v$ and let $w$ be the first vertex in $B$ of an $x - B$ path. 
	Then,
	$$
	e(v') \geq d(v',x) = \begin{cases}
		d(v,x) - 1 = e(v) - 1 & \text{ if } w = v' \\
		d(v,x) = e(v) & \text{ if }  w \neq v,v' \\
		d(v,x) + 1 = e(v) + 1 & \text{ if } w = v
	\end{cases}. 
	$$ 
\end{lemma}

\begin{proof}
	Note that for every vertex $b \in B$, a shortest path from $b$ to $x$ passes through $w$. 
	Hence, $d(b,x) = d(b,w) + d(w,x)$. 
	If $w = v' $, then $d(v,w) = 1$ and $d(v',w) = 0$, therefore $$d(v,x) = d(v,w) + d(w,x) = 1 + d(v',x),$$ 
	implying $d(v',x) = d(v,x) - 1$. 
	If $w \neq v,v'$, then  $d(v,w) = d(v',w) = 1$, hence $$d(v,x) = 1 + d(w,x)  = d(v',x).$$ 
	Eventually, if $w = v$, then $d(v,w) = 0$ and $d(v',w) = 1$, yielding $$d(v',w) = 1 + d(w,x) = 1 + d(v,x),$$ 
	and concluding the proof.
\end{proof}

Now, we can state the main result of this subsection.

\begin{theorem}\label{thm:complete_block}
	Let $G$ be a graph such that its centre is contained in a complete block $B$ and is not a cut vertex. 
	Then $Z(G) = B$. 
	In particular, the centre of a block graph is either a block or a cut vertex.
\end{theorem}

\begin{proof}
	Assume that $Z(G)$ is not a cut vertex. 
	By~\Cref{thm:centre_block}, there is a unique block $B$ containing it. 
	Hence we assume $B$ is complete and we want to show that $Z(G) = B$. 
	This is enough, since then the last assertion will then follow from the definition of block graphs. 
	
	First of all, if $G$ is 2-connected, then $G = B_Z$, and the eccentricity of all of its vertices is 1 by assumption. 
	Hence $Z(G) = G$, as desired. 
	
	From now on, let us assume that $G$ is not 2-connected. 
	We denote by $r$ the radius of $G$, recall it is equal to $r = e(x)$ for some $x\in Z(G)$.
	
	First, assume that the centre contains a vertex $u \in B$ that is not a cut vertex. 
	By~\Cref{lem:compare_e_internal_cut}, for every cut vertex $v \in B$, we have $e(v) \leq e(u)$, which implies that $v \in Z(G)$. 
	Moreover, for every vertex $u'$ in $B$ that is not a cut vertex, $u$ and $u'$ are twins (i.e. $N(u)\setminus\{u'\} = N(u')\setminus\{u\}$), which implies that they have the same eccentricity, so $u' \in Z(G)$. 
	Therefore, $B = Z(G)$.
	
	Now assume that every vertex in the centre is a cut vertex. 
	By assumption, this implies that $Z(G)$ is not of size 1, so there are at least two distinct cut vertices $v, v' \in B$ in $Z(G)$. 
	
	First, we show that every vertex in $B$ is a cut vertex. 
	For the sake of contradiction, let $u \in B$ such that $u$ is not a cut vertex. 
	By~\Cref{lem:compare_e_internal_cut}, we have that $e(u) \geq e(v)$, a contradiction with the fact that $u \notin Z(G)$. 
	This shows that every vertex of $B$ is a cut vertex.
	
	Now let us show that $Z(G)$ contains all cut vertices (thus all vertices) of $B$. 
	For the sake of contradiction, assume there is a cut vertex $v'' \in B\setminus Z(G)$. 
	That is, we have $e(v'') > r = e(v) = e(v')$. 
	Let $x$ be an eccentric vertex for $v''$ and let $w$ be the first vertex in $B$ of an $x-B$ path. 
	By~\Cref{lem:compare_e_cut_cut}, since $e(v'') > e(v)$, we have that $w = v$. 
	Now, applying the same lemma to $v'$ and $v''$, knowing that $w =v\neq v', v''$,  yields
	$$e(v') \geq e(v'') > e(v) = e(v'),$$ 
	a contradiction. 
	Thus, all cut vertices of $B$ are in $Z(G)$, and $Z(G) = B$.
\end{proof}

Note that by the theorem above, if $G$ is a connected block graph, then $(G, Z(G))$ is an anchored graph. 

\section{Quantum properties of block graphs}\label{sec:block_graphs}

In this section, we apply the techniques developed to compute the quantum automorphism groups of block graphs and to describe their quantum properties for the properties introduced in \Cref{subsec:superrigid_tractable}. 

We first need some terminology. 
Let $\mathcal E$ be a family of groups. 
We denote by $\Jor(\mathcal E)$ the closure of $\mathcal E$ under the operations of cartesian product and wreath product with permutation groups. 
Now let $\Ff$ be a family of compact quantum groups. 
We let $\Jor^+(\Ff)$ denote the closure of $\Ff$ under the operations of free products and free wreath product with quantum permutation groups. 

We show in this section that the family of quantum automorphism groups of block graphs and block-cographs are equal to $\Jor^+(\C)$ and that both families are superrigid tractable. 

These results can be seen as generalisations of previously known results. In~\cite{vanDobbenetal2023} and~\cite{Meunier2023}, it is shown that the family of quantum automorphism groups of trees is equal to $\Jor^+(\C)$, and in~\cite{Meunier2023}, it is shown that the family of tree-cographs is superrigid tractable and that the family of their quantum automorphism groups is $\Jor^+(\C)$. Since block graphs and block-cographs are strict supersets of trees and tree-cographs respectively, we generalise these results.  

\subsection{Rooted block graphs}

In order to obtain our results, we first need as a technical step to study the case of rooted block graphs. 

The following lemma allows us to reduce some of our proofs to the case of connected graphs. 

\begin{lemma}\label{lem:sums_rooted}
	If two rooted graphs are quantum isomorphic, then they have the same number of connected components, and they are two-by-two quantum isomorphic as rooted graphs.
\end{lemma}

\begin{proof}
	This follows from~\Cref{thm:partitioned_graphs} by noticing that a quantum isomorphism of rooted graphs from $(G,R)$ to $(H,S)$ preserves the connected components of $G$ and $H$ by~\cite[Lemma 3.4]{freslon2025blockstructures} and that the blocks $U[\varphi(D),D]$ are quantum isomorphisms of rooted graphs.
\end{proof}

We can now show the superrigidity of rooted block graphs.

\begin{lemma}\label{lem:qi_superrigid_rooted}
	Let $(G,R)$ be a rooted block graph quantum isomorphic to a rooted graph $(H,S)$. 
	Then $(G,R)\simeq (H,S)$. 
	In other words, rooted connected block graphs are superrigid among rooted graphs and thus satisfy (QI).
\end{lemma}

\begin{proof}
	By~\Cref{lem:sums_rooted}, it is enough to treat the case where $G$ and $H$ are both connected. Let $R=\{r\}$ and $S=\{s\}$.
	
	Let us define the following anchored graphs. If $r$ is a cut vertex of $G$ (and thus $s$ a cut vertex of $H$), define $R'=R$ and $S'=S$. If not, let $B_r$ be the unique block containing $r$ and $B_s$ be the unique block containing $s$, and set $R'=B_r$ and $S'=B_s$. In both cases, $(G,R')$ and $(H,S')$ are connected anchored graphs.
	
	We claim that $(G,R')$ and $(H,S')$ are quantum isomorphic as anchored graphs. In case where $S$ and $R$ are cut vertices, there is nothing to prove. So, let us assume that it is not the case and let $U$ be a quantum isomorphism of rooted graphs from $(G,R)$ to $(H,S)$. 
	All there is to prove is that $U = \Diag(U_0,U_1)$, with $U_1$ indexed by $S\times R$. 
	Let $x \in V(G)$ and $a\in V(H)$, and assume that $u_{ax}\neq 0$. 
	Since $u_{sr} = 1$, we have $0\neq u_{xa}u_{sr}$. 
	Hence, if $x\in R$, it is in the same block as $r$, and by~\cite[Lemma 3.5]{freslon2025blockstructures} $a$ is in the same block as $s$, which is $S$. 
	Similarly, if $a\in S$, then $x\in R$, which completes the proof of the claim.
	
	Hence by Theorem 5.2 of~\cite{freslon2025blockstructures} $(G,R')$ and $(H,S')$ have isomorphic rooted block structures with quantum isomorphic blocks. 
	Since the blocks of $G$ are complete graphs, and complete graphs clearly form a superrigid family, the blocks of $H$ are isomorphic to their corresponding blocks in $G$. 
	Since the blocks are complete, one can glue back these isomorphisms and obtain an isomorphism from $G$ to $H$ sending $R'$ to $S'$. 
	
	Now if $r$ and $s$ are cut vertices, we have the desired isomorphism of rooted graphs. 
	Otherwise, $r$ is an internal vertex of $R'$, which induces a complete graph. Hence up to automorphism, there is an isomorphism from $(G,r)$ to $(H,s)$. 
	This concludes the proof.
\end{proof}

Next, we study the quantum and classical automorphism groups of rooted block graphs.

\begin{lemma}\label{thm:qu_rooted_block_graphs}
	Let $(G,r)$ be a connected rooted block graph on $n \geq 2$ vertices. Then, there exists $k \in \N$ and for each $i \in \{1,\dots,k\}$, there exists a positive integer $a_i$ and a connected rooted block graph $(G_i,r_i)$ on strictly less than $n$ vertices that are two-by-two not isomorphic such that:
	$$
	\Qu(G,r) = \Qu\left(\bigoplus_{i=1}^k a_i.(G_i,r_i)\right)
	\text{ and }
	\Aut(G,r) = \Aut\left(\bigoplus_{i=1}^k a_i.(G_i,r_i)\right).
	$$
	Moreover, if $\bigoplus_{i=1}^k a_i.(G_i,r_i)$ satisfies Schmidt's criterion, so does $(G,r)$. 
\end{lemma}

\begin{proof}
	First, assume that $r$ is a vertex of degree 1 in $G$ and let $r'$ be its unique neighbour. Note that $\Qu(G,r) = \Qu(G\setminus \{r\},r')$ and that $\Aut(G,r) = \Aut(G\setminus \{r\},r')$. Moreover, $G'$ has strictly less vertices than $G$ does. The second part of the statement is clear. 
	
	So, we may assume that $r$ has degree at least 2. To prove the theorem, it is enough to find a disconnected rooted block graph $(H,R)$ such that each of its connected components has at most $n-1$ vertices, $\Qu(H,R)=\Qu(G,r)$, $\Aut(H,R)=\Aut(G,r)$, and that the second part of the statement holds. 
	
	If $r$ is a cut vertex, set $(H,R) = \Psi(G,r)$. Then, $\Qu(H,R)=\Qu(G,r)$ by \Cref{thm:iso}. In this case, it is also clear that $\Aut(H,R)=\Aut(G,r)$. For the second part of the statement, note that the natural classical isomorphism $\phi: \Aut(\Psi(G,r)) \to \Aut(G,r)$ satisfies $\Supp(\phi(f)) = \Supp(f) \setminus R \subseteq \Supp(f)$. Hence, if $f$ and $g$ are two nontrivial automorphisms of $\bigoplus_{i=1}^k a_i.(G_i,r_i)$ with disjoint support, then $\phi(f)$ and $\phi(g)$ are nontrivial automorphisms of $(G,r)$ with disjoint support. Thus the second part holds. 
 	
	If $r$ is an internal vertex of a block $b$, then since the degree of $r$ is at least 2, we have that $b'=b\setminus \{r\}$ is a block of~$G'=G\setminus \{r\}$. We set $(H,R) = \Psi(G', b')$ and we have that $\Qu(H,R) = \Qu(\Psi(G',b'))$ by~\Cref{thm:iso}. 
	It remains to show that $\Qu(G,r) = \Qu(G',b')$. 
	Let $\Qu(G,r) = (A,U)$ and $\Qu(G',b') = (B,W)$. 
	We have $U = \Diag(U_0,1)$ with $U_0$ a magic unitary adapted to $G'$, hence there is a (unique) unital $*$-morphism $f \colon B\to A$ such that $f(w_{xy}) = u_{xy}$ for any $x$, $y\in V(G') = V(G)\setminus\{r\}$. 
	Moreover, it is clear that $f$ is a quantum group morphism. 
	Now let $V = \Diag(W,1)$ be indexed by $V(G)$, with 1 corresponding to $v_{rr}$. 
	We claim that $V$ commutes with $M = \Adj(G)$. 
	Indeed, let us write $M = \begin{pmatrix}
		M_0 & {}^tL\\
		L & 0
	\end{pmatrix}$ where $M_0 = \Adj(G')$ and $L$ is indexed by $V(G')\times \{r\}$. 
	We have $WM_0 = M_0W$ by definition of $W$. 
	Let $x\in V(G')$, since $b$ is complete, we have:
	\[ [LW]_{rx} = \sum_{y\in V(G')} l_{ry}w_{yx} = \sum_{y\sim r} w_{yx} = \sum_{y\in b'} w_{yx}.\]
	Since $W$ stabilises $b'$, we have that $w_{yx} = 0$ if $y\in b'$ and $x\notin b'$, and $\sum_{y\in b'} w_{yx} = 1$ if $x\in b'$. 
	Hence we have shown that $[LW]_{rx} = l_{rx}$, which implies that $LW = L$. 
	Taking transpose, we obtain that ${}^tW{}^tL = {}^tL$, and since $W$ is a unitary we reach ${}^tL = W{}^tL$. 
	Hence $V$ commutes with $M$, as desired. 
	Since it clearly fixes $r$, we have a (unique) unital $*$-morphism $h\colon A\to B$ such that $h(u_{xy}) = v_{xy}$. 
	It is straightforward to check that $f$ and $h$ are inverse of one another, hence $f$ is a quantum group isomorphism, as desired.
	
	The classical case is straightforward.
	The second part of the statement follows from a similar argument as above, as well as noticing that for every $g \in \Aut(G',b')$, one can extend $g$ to an automorphism of $(G,r)$ while keeping the same support.
\end{proof}

Thanks to the previous result, we can prove that rooted block graphs satisfy the Schmidt alternative. We start with a general lemma.

\begin{lemma}\label{lem:sum_sa}
	Let $(G_1,r_1),\ldots,(G_k,r_k)$ be rooted connected graphs two-by-two not quantum isomorphic and assume that the family $\{(G_1,r_1),\ldots,(G_k,r_k)\}$ satisfies the Schmidt alternative.
	Let $a_1,\ldots,a_k\in \N$, and let $(G,R) = \bigoplus_{i=1}^k a_i.(G_i,r_i)$. 
	If $(G,R)$ does not satisfy Schmidt's criterion, then $(G,R)$ does not have quantum symmetry. 
\end{lemma}

\begin{proof}
	By~\Cref{thm:qu_sums_rooted_graphs}, we have $\Aut(G,R) = \prod_{i=1}^k \Aut(G_i,r_i)\wr S_{a_i}$. 
	Hence, it is relatively easy to see that since $(G,R)$ does not satisfy Schmidt's criterion, there is an integer $1\leq l\leq k$ such that;
	\begin{itemize}
		\item for all $i\neq l$, we have $\Aut(G_i,r_i) = 1$ and $a_i=1$,
		\item $a_l\leq 3$,
		\item if $a_l>1$, then $\Aut(G_l,r_l) = 1$,
	\end{itemize}
	where $1$ stands for the trivial group. 
	Since the $(G_i,r_i)$ satisfy the Schmidt alternative (SA), and (SA) implies (QA), we have that $\Qu(G_i,r_i) = \C$ when $i\neq l$. 
	Hence $\Qu(G,R) = \Qu(G_l,r_l)\wr S_{a_l}^+$ by~\Cref{thm:qu_sums_rooted_graphs}. 

	If $a_l>1$, then $\Aut(G_l,r_l) = 1$, and by assumption $\Qu(G_l,r_l) = \C$. 
	Hence $\Qu(G,R) = S_{a_l}^+$, whose algebra is commutative since $a_l\leq 3$. Therefore, $(G,R)$ does not have quantum symmetry. 
	
	If $a_l=1$, then $\Qu(G,R) = \Qu(G_l,r_l)$. 
	Since $(G_l,r_l)$ does not satisfy Schmidt's criterion (otherwise $(G,R)$ would), by the Schmidt alternative, $(G_l,r_l)$ does not have quantum symmetry, hence $C(\Qu(G_l,r_l)) = C(\Qu(G,R))$ is commutative. Hence, $(G,R)$ does not have quantum symmetry. 
\end{proof}

\begin{lemma}\label{thm:rooted_bg_sa}
	Rooted block graphs satisfy the Schmidt alternative. 
\end{lemma}

\begin{proof}
	Notice that by~\Cref{lem:sum_sa,lem:qi_superrigid_rooted} it is enough to prove it for connected rooted block graphs. 
	We prove it by induction on the number of vertices, it is clear for $K_1$. 
	So let us assume that all connected rooted block graphs on at most $n$ vertices satisfy the Schmidt alternative for some $n\geq 1$. 
	Let $(G,r)$ be a connected rooted block graph on $n+1$ vertices not satisfying Schmidt's criterion. 
	By~\Cref{thm:qu_rooted_block_graphs}, for some $k\geq 1$, there exit $(G_1,r_1),\ldots,(G_k,r_k)$, connected rooted block graphs on strictly less than $n+1$ vertices, two-by-two not isomorphic (hence not quantum isomorphic by \Cref{lem:qi_superrigid_rooted}), together with $a_1,\ldots,a_k\geq 1$, such that:
	$$
	\Qu(G,r) = \Qu\left(\bigoplus_{i=1}^k a_i.(G_i,r_i)\right)
	\text{ and }
	\Aut(G,r) = \Aut\left(\bigoplus_{i=1}^k a_i.(G_i,r_i)\right).
	$$
	Set $(H,S)=\bigoplus_{i=1}^k a_i.(G_i,r_i)$. By \Cref{thm:qu_rooted_block_graphs}, $(H,S)$ does not satisfy Schmidt's criterion. By induction hypothesis, we can apply \Cref{lem:sum_sa} to conclude that $(H,S)$ does not have quantum symmetry. Since $C(\Qu(G,r)) = C(\Qu(H,S))$, we have that $C(\Qu(G,r))$ is commutative and that $(G,r)$ does not have quantum symmetry.
\end{proof}

\subsection{Block graphs}

We now reach our main goal and describe the quantum properties of block graphs.

We start by an observation about block graphs naturally forming anchored graphs.

\begin{proposition} \label{prop:block_anchored}
	Let $G$ be a block graph and $Z=Z(G)$ be its centre. Then, $(G,Z)$ is an anchored graph. Moreover, $\Qu(G) = \Qu(G,Z)$. 
\end{proposition}
\begin{proof}
	The fact that $(G,Z)$ is an anchored graph follows directly from \Cref{thm:complete_block}. The second part follows from \Cref{prop:every_graph_anchored} along with \Cref{thm:complete_block}. 
\end{proof}

The following lemma relates isomorphism and $\Psi$ in the case where the anchors induce complete graphs. 
\begin{lemma}\label{lem:block_graphs_psi}
	Let $(G,Q)$ and $(H,S)$ be two connected anchored graphs on the same number of vertices. Moreover, assume that $Q$ and $S$ induce complete graphs.
	If $\Psi(G, Q)\simeq \Psi(H,S)$, then $G\simeq H$.
\end{lemma}
\begin{proof}
	Note that for an anchored graph $(F,R)$, we have $\abs{V(\Psi(F))} = \abs{V(F)}$ if and only if $R$ is a block of $F$. Hence, $Z(G)$ is a block of $G$ if and only if $Z(H)$ is a block of $H$ as well. 
		
	In the case where the centres are cut vertices, the proof is straightforward. 
	In the case where the centres are blocks, since both centres are complete graphs of the same size, by adding back the edges, the same isomorphism between $\Psi(G, Q)$ and $\Psi(H,S)$ is an isomorphism between $G$ and~$H$.
\end{proof}

We can now obtain the theorem on the quantum properties of block graphs, in the sense of \Cref{subsec:superrigid_tractable}.
\begin{theorem} \label{thm:block_graphs1}
	Block graphs form a superrigid tractable class of graphs.
\end{theorem}
\begin{proof}
	Since (SR) implies (QI) and (SA) implies (QA) , we only need to prove (SR) and (SA). 
	
	We start by showing that block graphs are superrigid.
	By~\cite[Lemma 5.15]{Meunier2023}, it is enough to show it only for connected block graphs. 
	Let $G$ be a connected block graph and let $H$ be a graph quantum isomorphic to $G$. 
	First, notice that by~\cite[Theorem 5.3]{freslon2025blockstructures}, there is a one-to-one correspondence between the blocks of $G$ and the blocks of $H$ such that the corresponding blocks are quantum isomorphic. Since the blocks of $G$ are complete and complete graphs satisfy (SR), every block of $H$ is complete as well. So, $H$ is a block graph.  
	
	Moreover, by~\cite[Lemma 3.7]{freslon2025blockstructures} we have that $(G,Z(G)) \qi (H,Z(H))$. 
	Now, by~\Cref{lem:psi_u}, we have that $\Psi(G,Z(G)) \qi \Psi(H,Z(H))$. 
	Since $\Psi(G,Z(G))$ is a rooted block graph, by~\Cref{lem:qi_superrigid_rooted}, we have $\Psi(G,Z(G)) \simeq \Psi(H,Z(H))$. Thus, we have $G\simeq H$ by~\Cref{lem:block_graphs_psi}, as desired.

	Now, let us show that block graphs satisfy the Schmidt alternative. 
	Again, by~\cite[Theorem 5.6]{Meunier2023}, it is enough to show it for connected block graphs only. 
	Let $G$ be a connected block graph not satisfying Schmidt's criterion. Since every automorphism of $G$ preserves the centre, $(G,Z(G))$ does not satisfy Schmidt's criterion either. 
	With the same argument as in the proof of \Cref{thm:qu_rooted_block_graphs}, the natural isomorphism $\phi \colon \Aut(\Psi(G,Z(G))) \to \Aut(G,Z(G)) $, we know that $ \Psi(G,Z(G)) $ does not satisfy Schmidt's criterion either. Thus, by \Cref{thm:rooted_bg_sa}, $ \Psi(G,Z(G)) $ does not have quantum symmetry. Now, by \Cref{thm:iso}, we have that $(G,Z(G))$ does not have quantum symmetry. Finally, by \Cref{prop:block_anchored}, we conclude that $G$ does not have quantum symmetry.
\end{proof}

\begin{corollary}
	0-hyperbolicity is preserved under quantum isomorphism. 
\end{corollary}
\begin{proof}
	Immediate from the fact that block graphs are superrigid and that 0-hyperbolic graphs are exactly block graphs.
\end{proof}

Let $G$ be a connected block graph. Notice that the isolated vertices of $\Psi(G,Z(G))$ are exactly the internal vertices of $Z(G)$. Let $z$ be the  number of internal vertices of $Z(G)$ (thus possibly $z=0$). Since block graphs are superrigid by \Cref{thm:block_graphs1}, we can write $\Psi(G,Z(G)) = z.(K_1,V(K_1)) \oplus  \left( \bigoplus_{i=1}^k a_i.(G_i,r_i) \right) $ where the graph $(G_i,r_i)$ are connected rooted block graphs, two-by-two not isomorphic. In what follows, we consider $S_0$ and $S_0^+$ to be trivial. 

\begin{theorem} \label{thm:block_graphs2}
	Let $G$ be a connected block graph. Writing 
	$$\Psi(G,Z(G)) =  z.(K_1,V(K_1)) \oplus  \left( \bigoplus_{i=1}^k a_i.(G_i,r_i) \right)$$ as above, we have:
	$$
	\Qu(G) = S_z^+ \ast \bigast_{i=1}^k \left( \Fix_{\Qu(G_i)}(r_i) \wr S_{a_i}^+\right) \text{ and } \Aut(G) = S_z \ast \prod_{i=1}^k \left( \Fix_{\Aut(G_i)}(r_i) \wr S_{a_i} \right).
	$$	
	Moreover, the family of automorphism groups and quantum automorphism groups of block graphs are respectively the families $\Jor(1)$ and $\Jor^+(\C)$. 
\end{theorem}

\begin{proof}
	
	The first formula follows from~\Cref{thm:decomposition_qu_group_anchored} and the second one is straightforward. 
	Let $\mathcal E$ be the family of automorphism groups of block graphs, and let $\Ff$ be the family of quantum automorphism groups of block graphs. 
	
	Let us show that $\Ff \subset \Jor^+(\C)$. 
	By~\cite[Theorem 5.3]{Meunier2023}, it is enough to show that $ \Qu(G) \in \Jor^+(\C) $ for every connected block graph $G$.
	First, we claim that $ \Qu(H,s) \in \Jor^+(\C) $ for every connected rooted block graph $(H,s)$. This follows from a straightforward induction, using \Cref{thm:qu_rooted_block_graphs} and \Cref{thm:qu_sums_rooted_graphs}. Now, let $G$ be a connected block graph. Writing $\Psi(G,Z(G))$ as above, we have that $\Fix_{\Qu(G_i)}(r_i) = \Qu(G_i,r_i)$. Therefore, by the formula we just proved, $\Qu(G) \in \Jor^+(\C) $, as desired. 
	
	Now, note that every tree is a block graph, and by~\cite[Theorem 6.30]{Meunier2023}, every compact quantum group in $\Jor^+(\C)$ is the quantum automorphism group of a tree. Thus, $\Jor^+(\C) \subseteq \Ff$. 
	
	For the classical automorphism group, a similar argument shows that $\mathcal E \subseteq \Jor(1)$. Moreover, from~\cite{jordan1869assemblages} (see the arguments on page 188, particularly Equation (4)), evert group in $\Jor(1)$ is the automorphism group of a tree, thus of a block graph, proving $\Jor(1) \subseteq \mathcal E $.
\end{proof}

\subsection{Block-cographs}

In this section, we extend our result to a larger class of graphs called block co-graphs. 

Given a family $\Ff$ of graphs, the family $\co(\Ff)$ is the smallest family that contains $\Ff$ and is stable by taking complement and disjoint union. Starting from $\Ff = \{K_1\}$, one finds back the definition of cographs, which is why the class $\co(\Ff)$ is called \emph{$\Ff$-cographs}.

We call \emph{block-cographs} the family $\Ff$-cographs where $\Ff$ is the family of block graphs. Block-cographs form a strict superclass of block graphs. For instance, a cycle of length 4 is a block-cograph but not a block graph. 

To prove our result, the idea is to use the following theorem which is a special case of Theorems 5.13, 5.16, 5.19, and 5.22 of~\cite{Meunier2023}.
\begin{theorem} \label{thm:cograph_closure}
	If $\Ff$ is a superrigid tractable family of graphs such that 
	\begin{itemize}
		\item $\Ff$ is stable by taking connected components; 
		\item $\Ff$ is stable by taking complement;
		\item $K_1 \in \Ff$.
	\end{itemize}
	Then, $\co(\Ff)$ is superrigid tractable. Moreover, $$\Qu(\co(\Ff)) = \Jor^+(\Qu(\Ff))  \text{ and } \Aut(\co(\Ff)) = \Jor(\Aut(\Ff)).$$
\end{theorem}

Using this theorem, we generalise our former result to block-cographs. 
\begin{theorem} \label{thm:blockcographs}
	The family of block-cographs is superrigid tractable. Moreover, the family of their automorphism groups and their quantum automorphism groups are respectively $\Jor(1)$ and $\Jor^+(\C)$. 
\end{theorem}
\begin{proof}
	Let $\Ff$ be the class of block graphs and their complements. We prove that $\Ff$ satisfies the conditions of \Cref{thm:cograph_closure}. The superrigidity follows from the fact that two graphs are quantum isomorphic if and only if their complements are quantum isomorphic. The (SA) axiom follows from the fact that a graph and its complement have the same classical and quantum automorphism group. Thus, $ \Ff $ is superrigid tractable. Moreover, it clearly is stable by taking competent and it contains $K_1$. Finally, note that for any two graphs $H$ and $G$, if $H$ is an induced subgraph of $G^c$, then $H^c$ is an induced subgraph of $G$. Now, since block graphs are stable under taking connected components, so is $\Ff$.
	
	Hence, the result follows from \Cref{thm:cograph_closure}, and the fact that $\co(\Ff)$ is the same family as block-cographs. 
\end{proof}

It is worth noting that the family of block-cographs is strictly bigger than the family of tree-cographs, for instance the bull graph (see \Cref{fig:bull}) is a block-cograph (and even a block graph) but not a tree-cograph (since it is self-complement).
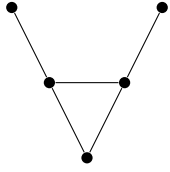
\begin{figure}
	\begin{tikzpicture}
		\tikzset{bullet/.style={circle, fill=black, inner sep=1.5pt}}
		\node[bullet] (1) at (0,0) {};
		\node[bullet] (2) at (-1/2,1) {};
		\node[bullet] (3) at (1/2,1) {};
		\node[bullet] (4) at (-1,2) {};
		\node[bullet] (5) at (1,2) {};
		\draw (4) -- (2) -- (3) -- (1) -- (2);
		\draw (5) -- (3) ;
	\end{tikzpicture}
	\caption{The bull graph} \label{fig:bull}
\end{figure}

\end{document}